\documentclass[11pt]{amsart}
\usepackage{amssymb, amsthm, amsmath, slashed, wasysym, epstopdf, graphicx}
\usepackage{bbm}
\usepackage{mathrsfs}
\usepackage{color}

\newtheorem{theorem}{Theorem}
\newtheorem{lemma}{Lemma}

\newtheorem{proposition}{Proposition}

\theoremstyle{definition}
\newtheorem{definition}{Definition}

\theoremstyle{definition}
\newtheorem{remark}{Remark}

\theoremstyle{definition}
\newtheorem{example}{Example}

\newcommand{ \rn }[1]{\mathbb{R}^{#1}}

\title[The Broken Ray Transform in $n$ Dimensions]{The Broken Ray Transform in $n$ Dimensions with Flat Reflecting Boundary}

\author[M. Hubenthal]{Mark Hubenthal}
\address{Department of Mathematics, 641 PGH, University of Houston}
\email{hubenjm@math.uh.edu}
\subjclass[2010]{45Q05, 47G30, 53C65}
\keywords{inverse problems, integral geometry, microlocal analysis}

\begin{document}
\maketitle

\begin{abstract}
We study the broken ray transform on $n$-dimensional Euclidean domains where the reflecting parts of the boundary are flat and establish injectivity and stability under certain conditions. Given a subset $E$ of the boundary $\partial \Omega$ such that $\partial \Omega \setminus E$ is itself flat (contained in a union of hyperplanes), we measure the attenuation of all broken rays starting and ending at $E$ with the standard optical reflection rule applied to $\partial \Omega \setminus E$. By localizing the measurement operator around broken rays which reflect off a fixed sequence of flat hyperplanes, we can apply the analytic microlocal approach of Frigyik, Stefanov, and Uhlmann (\cite{xraygeneric}) for the ordinary ray transform by means of a local path unfolding. This generalizes the author's previous result in \cite{hubenthal2}, although we can no longer treat reflections from corner points. Similar to the result for the two dimensional square, we show that the normal operator is a classical pseudo differential operator of order $-1$ plus a smoothing term with $C_{0}^{\infty}$ Schwartz kernel.
\end{abstract}

\section{Introduction \label{sec:intro}}
In this work, we focus on a particular variant of the attenuated x-ray transform which adds another layer of complexity by incorporating billiard trajectories into the problem. The goal will be to establish injectivity and stability results for such a transform analogous to those already existing for the much more familiar x-ray transform. Recall that the x-ray transform and its generalizations have long received attention from mathematicians, partly due to its utility in tackling other inverse problems, and partly for its own geometrical interest. The standard x-ray transform of a function $f$ defined on $\rn{n}$ can be written as
\begin{equation*}
Xf(\gamma) = \int f(\gamma(t))\,dt, \quad \gamma \in \mathcal{F}
\end{equation*}
where $\mathcal{F}$ is the collection of all lines in $\rn{n}$. There are various inversion formulas known in the Euclidean setting, many of which involving the Hilbert transform. However, the one most relevant to the approach used in this paper is the following:
\begin{equation*}
f = c_{n}(-\Delta)^{1/2}X^{*}Xf, \quad \forall f \in \mathcal{E}'(\rn{n}).
\end{equation*}
Here $X^{*}$ is the adjoint to $X$, and $c_{n}$ is a constant depending on the dimension and $\Delta$ is the Laplacian, see \cite{nattererbook}.

Perhaps more relevant to the microlocal approach of this paper, it is well-known that
\begin{equation*}
X^{*}Xf(x) \simeq f  * \frac{1}{|x|^{n-1}} = (2\pi)^{-n} \int e^{i(x-y)\cdot \xi} f(y) |\xi|^{-1}\,dy\,d\xi.
\end{equation*}
This means $X^{*}X$ is a pseudodifferential operator of order $-1$ (its symbol is $|\xi|^{-1}$) that is elliptic on $\rn{n}$. 

A bit more difficult to work with is the attenuated or weighted x-ray transform given by
\begin{equation*}
X_{w}f(\gamma) = \int w(\gamma(t),\gamma'(t))f(\gamma(t))\,dt.
\end{equation*}
In many cases, as is the case in this paper, the weight $w$ is an exponential function induced by an attenuation $\sigma$. In \cite{novikov,novikov2}, Novikov presents an inversion formula for the attenuated x-ray transform in $2$ dimensions when $\sigma$ is isotropic and then derives specific range conditions. \cite{natterer} soon after showed an  inversion formula for the attenuated Radon transform which is equivalent to Novikov's in the $2$-dimensional case. \cite{boman} also presents an inversion formula for the attenuated Radon transform using a different approach. Bal later derived in \cite{balradon} a specific reconstruction scheme based on the inversion formula of Novikov, which exploits some redundancies in the data and even considers the case of an angularly varying source. We refer the reader also to \cite{finch1,finch2, quinto,quinto2} for more background on the x-ray transform.

It should also be mentioned that much work has been done with the x-ray transform on manifolds (in particular, Frigyik, Stefanov, and Uhlmann in \cite{xraygeneric}). One can also consider the transform applied to tensors which has been treated extensively in \cite{xraytensor,stefanov1,stefanovrio}, typically in the Riemannian case where the relevant family of curves consists of geodesics. More recently, some work has been done in \cite{causticxray} for the geodesic x-ray transform in the presence of fold caustics in the metric $g$. Also a recent result of Uhlmann and Vasy proved injectivity of the local geodesic x-ray transform in \cite{vasy} assuming a convexity condition on the boundary.

Perhaps the main motivation for studying the modified x-ray transform presented in this paper is the recent work of Kenig and Salo, \cite{salokenig} on the anisotropic Calder\'on problem with partial data. The authors' approach in that work led to a variation of the x-ray transform, which we call the broken ray transform. As a simple example, consider the unit square $\Omega = (0,1)^{2} \subset \rn{2}$, and let $E \subset \partial \Omega$ be the left edge $\{0\} \times [0,1]$. Then for each $(x,\theta) \in E \times \mathbb{S}^{1}$ such that $\theta \cdot e_{1} > 0$ (i.e. $\theta$ points inward), we let $\gamma_{x,\theta}$ be the piecewise linear curve starting at $x$ with initial direction $\theta$ and ending at the next intersection with $E$, such that whenever $\gamma_{x,\theta}$ intersects $\partial \Omega \setminus E$ its direction changes according to the standard rule of billiards. We then measure the integral of the unknown function $f$ over all such \textit{broken rays} $\gamma_{x,\theta}$. See Figure \ref{fig:brtscheme}.

We should also mention the work of Eskin in \cite{eskin}, which considered the Schrodinger equation with electric and magnetic potentials on a domain $\Omega$ in $\rn{2}$ with finitely many internal, convex obstacles. From knowledge of the Dirichlet to Neumann map at the boundary, one can recover the integrals of either potential along broken rays starting and ending on $\partial \Omega$ with reflections occurring on the boundaries of the internal obstacles inside. One key assumption however, is that there cannot be any trapped broken rays. In order to ensure this, the author adds corners to the interior obstacles' boundaries as necessary. From these assumptions, Eskin shows that one can uniquely recover the smooth electric and magnetic potentials from such integrals.

\begin{figure}
\centering
\def \svgwidth{0.4\columnwidth}
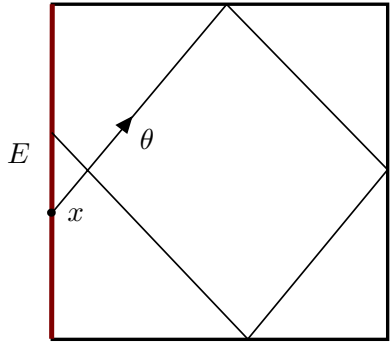
\caption{A broken ray on the square. \label{fig:brtscheme}}
\end{figure}

In \cite{hubenthal2} the author utilized a reflection approach to deduce injectivity and stability results for such a transform (possibly with a non constant weight) on the unit square analogous to those in \cite{xraygeneric}. A similar reflection approach is utilized in \cite{joonas2} which considers the geodesic broken ray transform on a particular class of Riemannian manifolds where the reflecting subset of the boundary is in the shape of a flat cone. Ilmavirta also recently proved in \cite{joonas} an injectivity result for the broken ray transform on the open disk $\mathbb{D}$ assuming the unknown function $f$ is uniformly quasianalytic in the angular variable (when written in polar coordinates). However, such an approach does not yield a stability estimate. At the time of this writing, it remains an open problem to determine whether the broken ray transform is injective on $L^{2}(\mathbb{D})$ for particular measurement subsets $E$ (e.g. if $E$ is an open arc).

The approach we use here will be similar to that used in \cite{hubenthal2}. In particular, we exploit the flatness of the reflecting boundaries to be able to unfold all broken rays in some neighborhood of a fixed broken ray within a certain augmented domain we will specifically construct. The problem then becomes a standard x-ray transform on the augmented domain. What follows is largely based on the microlocal analytic techniques of \cite{xraygeneric}. The key differences of this work and previous work on the broken ray transform in \cite{hubenthal2} however, is that here we make a more elegant change of variables in order to simplify the normal operator that is easily generalized to all dimensions. Furthermore, we cover a broader range of domains than those allowed in \cite{joonas} because we do not require there to be one unique augmented domain that applies for all broken rays. The augmentation of the domain $\Omega$ via the unfolding of a neighborhood of some fixed broken ray is dependent on the particular broken ray considered. In this sense, the approach here is more local.

The structure of this paper goes as follows. In \S \ref{sec:mainresults} we describe the problem and notations and then state the main results. \S \ref{sec:injective} applies the microlocal ideas of \cite{xraygeneric} and adapts the path unfolding technique of \cite{hubenthal2} to show the recovery of the analytic wavefront set of the unknown function $f$ when the corresponding weight function on the unfolded domain is analytic. This is always the case when the attenuation is identically $0$, but such an assumption is not necessary in general. Injectivity is then established under certain conditions on the available broken rays. In \S \ref{sec:stability} we then show the details required to prove the stability estimate of the inverse problem. Specifically \S \ref{sec:normalop}-\ref{sec:changevar} details how the normal operator decomposes into a pseudodifferential operator of order $-1$ plus a smoothing term. Finally, we extend the stability estimate and injectivity to $C^{2}$ perturbations of the attenuation $\sigma$ in \S \ref{sec:reducingsmoothness}.

\section{Statement of Main Results \label{sec:mainresults}}
Let $\Omega \subset \rn{n}$ be a convex domain with smooth boundary. Define $\Gamma_{-} = \{(x,\theta) \in \Omega \times \mathbb{S}^{n-1} \, | \, x \in \partial \Omega, \, \nu(x) \cdot \theta < 0\}$ as the set of ingoing unit vectors on $\partial \Omega$. Here $\nu(x)$ is the outward unit normal vector to $\partial \Omega$. We also define, for a general subset $E \subset \partial \Omega$,
\begin{equation}
\Gamma_{-}(E) := \{(x,\theta) \in \Omega \times \mathbb{S}^{n-1} \, | \, x \in E,  \, \nu(x) \cdot \theta  < 0\}.
\end{equation}
Typically, one might have $E$ to be an open set, but it is not important. Throughout this work, we will assume that $\partial \Omega \setminus E$ is contained in a union of hyperplanes (i.e. each component is flat).

We will call any unit speed curve $\gamma_{x,\theta}$ a \textit{broken ray} in $\Omega$ if
\begin{enumerate}
\item[(a)] $(\gamma_{x,\theta}(0), \dot{\gamma}_{x,\theta}(0)) = (x,\theta) \in \Gamma_{-}(E)$,\\
\item[(b)] it consists of finitely many line segments $\gamma_{x,\theta,1}, \gamma_{x,\theta,2}, \ldots, \gamma_{x,\theta,N}$ before hitting $E$ again,\\
\item[(c)] it obeys the geometrical optics reflection law whenever intersecting $\partial \Omega \setminus E$:
\begin{equation}
\dot{\gamma}_{x,\theta,j+1}(0) = \dot{\gamma}_{x,\theta,j}(L_{j}) - 2 \left( \nu(\gamma_{x,\theta,j}(L_{j})) \cdot \dot{\gamma}_{x,\theta,j}(L_{j}) \right)\nu(\gamma_{x,\theta,j}(L_{j})), \,1 \leq j \leq N.
\end{equation}
\end{enumerate}
Here $L_{j}$ denotes the length of $\gamma_{x,\theta,j}$. 

Let us establish some important notation. We use $T:\Gamma_{-} \to \Gamma_{-}$ to denote the billiard map taking a vector $(x,\theta) \in \Gamma_{-}$ to $(x',\theta') \in \Gamma_{-}$, where $x'$ is the intersection point of the line $\gamma_{x,\theta}$ with $\partial \Omega$, and $\theta'$ is the reflected direction. Given $(x,\xi) \in \Omega \times \mathbb{S}^{n-1}$, let $\tau_{\pm}(x,\theta) = \min \{ t > 0 \, | \, x \pm t\theta \in \partial \Omega\}$ and define the diameter through $x$ in the direction $\theta$ by
\begin{equation*}
\tau(x,\theta) := \tau_{-}(x,\theta) + \tau_{+}(x,\theta).
\end{equation*}
We use $\pi_{1}:\rn{n} \times \mathbb{S}^{n-1} \to \rn{n}$ and $\pi_{2}: \rn{n} \times \mathbb{S}^{n-1} \to \mathbb{S}^{n-1}$ to denote the standard projection operators onto $\rn{n}$ and $\mathbb{S}^{n-1}$, respectively.

We define the function $M(x,\theta)$ for $(x,\theta) \in \Gamma_{-}(E)$ as the number of reflections of the broken ray $\gamma_{x,\theta}$ before returning to $E$. The area form on $\Gamma_{-}$ is given by $d\Sigma = |\nu(x) \cdot \theta| \, dS(x)\, d\theta$, where $dS(x)$ is the surface measure on $\partial \Omega$. From \cite{sergebook} we have that $T$ preserves $d\Sigma$, even for domains with piecewise smooth boundary. That is, $T^{*}(d\Sigma) = d\Sigma$, where $T^{*}$ denotes the pullback of $T$.

Assume that the attenuation $\sigma$ satisfies $\sigma \in C^{\infty}(\Omega \times \mathbb{S}^{n-1})$. The broken ray transform with respect to $E\subset \partial \Omega$ and with attenuation $\sigma$, denoted by $I_{\sigma, E}:L^{2}(\Omega) \to L^{2}(\Gamma_{-}, d\Sigma)$, is defined as
\begin{align}
& I_{\sigma, E}f(x,\theta) \notag \\
& := \sum_{j=0}^{M(x,\theta)}\int_{\rn{+}} \exp\left( - \sum_{m=0}^{j-1}\int_{\rn{+}}\sigma(\pi_{1} \circ T^{m}(x, \theta) + \tau \pi_{2}\circ T^{m}(x, \theta), \pi_{2} \circ T^{m}(x, \theta))\,d\tau\right) \notag\\
& \qquad \cdot \exp\left(-\int_{\rn{+}}\sigma(\pi_{1} \circ T^{j}(x,\theta) + (t-\tau)\pi_{2} \circ T^{j}(x,\theta), \pi_{2} \circ T^{j}(x,\theta))\,d\tau\right)\notag\\
& \qquad \cdot  f(\pi_{1} \circ T^{j}(x,\theta) + t\pi_{2} \circ T^{j}(x,\theta))\,dt. \label{eq:forwardoperator}\\
& = \sum_{j=0}^{M(x,\theta)}\int_{\rn{+}} \left[w_{j}f\right](\pi_{1} \circ T^{j}(x,\theta) + t\pi_{2} \circ T^{j}(x,\theta), \, \pi_{2} \circ T^{j}(x,\theta))\,dt. \notag
\end{align}
for all regular broken rays $\gamma_{x,\theta}$. The weight functions $w_{j}$ on $\Omega \times \mathbb{S}^{n-1}$ are given by
\begin{align}
w_{j}(y,\eta) & = \exp\left( -\sum_{m=0}^{j-1} \int_{\mathbb{R}_{+}}\sigma( z_{m-j} + \tau \theta_{m-j}, \theta_{m-j})\,d\tau \right)\exp\left( -\int_{\mathbb{R}_{+}}\sigma(y - \tau \eta, \eta)\,d\tau \right)\notag\\
& = \exp\left( -\sum_{m=1}^{j} \int_{\mathbb{R}_{+}}\sigma( z_{-m}+ \tau \theta_{-m}, \, \theta_{-m})\,d\tau\right) \exp\left( -\int_{\mathbb{R}_{+}}\sigma(y - \tau \eta, \, \eta)\,d\tau \right).\notag
\end{align}
In the above definition, for convenience of notation we also extend all functions outside of $\Omega$ by zero.

Note that typically, $M(x,\theta)$ is piecewise constant and intuitively it will have jumps near broken rays that intersect $\partial E$. In two dimensions, one cannot ignore broken rays that intersect $\partial E$ and still obtain injectivity for functions supported on some subset of $\Omega$ (see \cite{hubenthal2}). However, if $\mathrm{supp}(f)$ is known to be contained in some compact subset of $\Omega$, then sometimes the extra broken ray segment introduced by a jump in $M(x,\theta)$ (or segment removed, respectively) might not intersect $\mathrm{supp}(f)$, which then implies that the broken ray transform does not introduce a singularity independent of $f$.  in higher dimensions, we can usually ignore such broken rays that intersect $\partial E$ since we have many more covectors to choose from at a given point in order to detect a singularity in a particular direction. We thus have the following definition of the particular broken rays we would like to restrict ourselves to:
\begin{definition}Let $E \subset \partial \Omega$ and $K \Subset \Omega$. We say that $\gamma_{x,\theta}$ is a \textit{regular broken ray} for $(x,\theta) \in \Gamma_{-}(E)$ with respect to $E$ and $K$ if there exists some smooth cutoff function $\alpha \in C_{0}^{\infty}(\Gamma_{-})$ with $\alpha(x,\theta) = 1$ such that $\alpha I_{\sigma, E}f \in C^{\infty}(\Gamma_{-})$ for all $f \in C^{\infty}(K)$.
\end{definition}

\begin{remark}A sufficient condition to be a regular broken ray is that $\gamma_{x,\theta}$ never touches a boundary point of $E$. However, depending on $E$ and $K$, it is possible to have broken rays touching boundary points which are still regular. In particular, for the square in two dimensions, if $E = \{0\} \times [0,1] \cup [0,1] \times \{0\} \cup [0,\epsilon) \times \{1\} \cup \{1\} \times [0,\epsilon)$ where $\epsilon = \frac{1}{2} \mathrm{dist}(K, \partial \Omega)$, then any broken ray that intersects $\partial \Omega$ near $(\epsilon, 1)$ either: (i) terminates (if the intersection point is in $E$); (ii) the next reflected segment or the current one is disjoint from $K$.  Note also that in this example every broken ray has at most $2$ reflections. Moreover, there is effectively only one reflection to consider for broken rays that pass near $\partial E$. It is important that the reflected segment disjoint from $K$ be either at the end of the beginning of the broken ray, because the weight function may be constant $1$. \end{remark}
Similar to Theorem (b) of \cite{inversesource}, we want to extend $\alpha I_{\sigma, E}$ to be well-defined on the space $L^{2}(\Omega)$ for any $\alpha \in C_{0}^{\infty}(\Gamma_{-})$ which limits the number of reflections, and also to be bounded. To this end, we need only show that the image is a well-defined $L^{2}$ function. Even though the proof is also given in Lemma 1 of \cite{hubenthal2}, we restate it here for convenience. Later we also see that it is possible to extend $\alpha I_{\sigma, E}$ to be defined on the space of compactly supported distributions on $\Omega$, $\mathcal{E}'(\Omega)$.
\begin{lemma}Let $\alpha \in C_{0}^{\infty}(\Gamma_{-})$ and suppose all broken rays in $\mathrm{supp}(\alpha)$ are regular and have at most $M_{max} \in \mathbb{N}$ reflections. Then $\alpha I_{\sigma, E}$ extends to a bounded operator from $L^{2}(\Omega \times \mathbb{S}^{n-1}) \to L^{2}(\Gamma_{-}, d\Sigma)$. \label{lemma:raybounded}\end{lemma}
\begin{proof}
First we recall an identity from \cite{inversesource} which asserts that for any function $f \in L^{2}(\Omega \times \mathbb{S}^{n-1})$, we have
\begin{equation*}
\int_{\Gamma_{-}}\int_{\mathbb{R}_{+}} f(x+t\theta, \theta)\,dt\,d\Sigma = \int_{\Omega \times \mathbb{S}^{n-1}} f(x,\theta)\,dx\,d\theta.
\end{equation*}
Now observe that
\begin{align*}
& \| \alpha I_{\sigma, E} f(x,\theta) \|_{L^{2}(\Gamma_{-})}^{2} \\
& \leq 2 \int_{\Gamma_{-}} |\alpha(x,\theta)|^{2}\sum_{j=0}^{M(x,\theta)}\left| \int_{\mathbb{R}_{+}} [w_{j} f](z_{j}(x,\theta) + t\theta_{j}(x,\theta), \theta_{j}(x,\theta))\,dt \right|^{2} \, d\Sigma\\
&\leq 2 \int_{\Gamma_{-}} \sum_{j=0}^{M_{max}}\left| \int_{\mathbb{R}_{+}} \chi_{[0,\tau_{+}(z_{j},\theta_{j})]}(t)[w_{j} f](z_{j}(x,\theta) + t\theta_{j}(x,\theta), \theta_{j}(x,\theta))\,dt \right|^{2} \, d\Sigma\\
& \leq 2 \int_{\Gamma_{-}} \sum_{j=0}^{M_{max}} \|\chi_{[0,\tau_{+}(z_{j},\theta_{j})]}\|_{L^{2}(\mathbb{R}_{+})}^{2}\|[w_{j}f](z_{j}+t\theta_{j},\theta_{j})\|_{L^{2}(\mathbb{R}_{+})}^{2}\,d\Sigma\\
& \leq 2 \int_{\Gamma_{-}} \sum_{j=0}^{M_{max}} \mathrm{diam}(\Omega)\|f(z_{j}+t\theta_{j},\theta_{j})\|_{L^{2}(\mathbb{R}_{+})}^{2}\,d\Sigma\\
& = 2 \int_{\Gamma_{-}} \sum_{j=0}^{M_{max}} \mathrm{diam}(\Omega)\int_{\mathbb{R}_{+}}|f(z_{j}+t\theta_{j},\theta_{j})|^{2}\,dt\,d\Sigma\\
& = 2 (M_{max}+1)\mathrm{diam}(\Omega)\int_{\Gamma_{-}} \int_{\mathbb{R}_{+}}|f(x+t\theta,\theta)|^{2}\,dt\,d\Sigma\\
& = 2 (M_{max}+1)\mathrm{diam}(\Omega)\|f\|_{L^{2}(\Omega \times \mathbb{S}^{n-1})}^{2}.
\end{align*}
\end{proof}

We can then think of $I_{\sigma, E}$ locally as an x-ray transform defined on a larger space obtained via reflection across a given sequence of hyperplanes defining $\partial \Omega \setminus E$ which corresponds to the sequence of reflection faces for a given broken ray $\gamma_{x_{0},\theta_{0}}$. Let $\alpha$ be a smooth cutoff function on $\Gamma_{-}$ which is equal to $1$ near $(x_{0},\theta_{0}) \in \Gamma_{-}(E)$ with $\gamma_{x_{0},\theta_{0}}$ a regular broken ray, and such that all broken rays in its support reflect on the same sequence of hyperplanes, denoted by $\{P_{1},\ldots, P_{N}\}$, where $N = N(x_{0}, \theta_{0})$ is the number of reflections of $\gamma_{x_{0},\theta_{0}}$. Representing each affine hyperplane $P_{j}$ by a pair $(a_{j}, \xi_{j}) \in \rn{n} \times \mathbb{S}^{n-1}$, where $a_{j} \in P_{j}$ and $\xi_{j}$ is a unit normal vector of $P_{j}$, we consider the operator of reflection across $P_{j}$ given by
\begin{equation}
R_{j}(x) = x + 2\xi_{j} (a_{j}-x)\cdot \xi_{j}.
\end{equation}

Now we can define an unfolded version $\widetilde{\gamma}_{x,\theta}$ of $\gamma_{x,\theta}$ as follows: if $\gamma_{x,\theta}$ consists of a collection of segments $\{\gamma_{x,\theta,0}, \ldots, \gamma_{x,\theta,N}\}$, then
\begin{equation}
\widetilde{\gamma}_{x,\theta} = \bigcup_{j=0}^{N}R_{0}\circ R_{1} \circ \cdots \circ R_{j}(\gamma_{x,\theta,j})
\end{equation}
Geometrically, it is easy to see since each component of $\partial \Omega \setminus E$ is flat that $\widetilde{\gamma}_{x,\theta}$ is a straight line segment in $\rn{n}$.

We then construct a domain $\widetilde{\Omega}$ which resembles a beam containing the unfolded broken ray $\widetilde{\gamma}_{x_{0},\theta_{0}}$. First we define the set
\begin{equation}
\Omega_{0} := \{ z \in \rn{n} \, | \,  z = \gamma_{x,\theta,0}(t), \, 0 \leq t < \tau_{+}(x,\theta), \, (x,\theta) \in \mathrm{supp}(\alpha)\}.
\end{equation}
Then for $1 \leq j \leq N$ we define
\begin{align}
\Omega_{j} & := R_{1} \circ R_{2} \circ \cdots \circ R_{j}\Big( \{ z \in \rn{n} \, | \, x = \gamma_{x,\theta,j}(t), \, 0 \leq t < \tau_{+}(T^{j}(x,\theta)), \notag \\
& \hspace{4cm}  (x,\theta) \in \mathrm{supp}(\alpha) \}\Big).
\end{align}
Finally,
\begin{equation}
\widetilde{\Omega} := \bigcup_{j=0}^{N}\Omega_{j},
\end{equation}
which resembles a closed beam of straight line segments with initial jet in $\mathrm{supp}(\alpha)$. Note that this construction \textit{depends} on the broken ray $\gamma_{x_{0},\theta_{0}}$ and is only valid for $(x,\theta)$ in some neighborhood of $(x_{0},\theta_{0})$. Also note that we cannot in general construct $\widetilde{\Omega}$ by reflecting the entire domain $\Omega$ repeatedly across the desired hyperplanes, since it is possible that the reflected versions of $\Omega$ will overlap (see Figure \ref{fig:extendedphasedomain2}). Finally, by reversing the reflection sequence, any covector $(\widetilde{z}, \widetilde{\xi}) \in T^{*}\widetilde{\Omega}$ corresponds to a unique covector $(z,\xi) \in T^{*}\Omega$.

We can define a distribution $\widetilde{f} \in \mathcal{D}'(\widetilde{\Omega})$ corresponding to $f \in \mathcal{E}'(\Omega)$ restricted to a neighborhood of $\gamma_{x_{0},\theta_{0}}$ by 
\begin{equation}
\langle \widetilde{f}, \phi \rangle_{\widetilde{\Omega}} = \sum_{j=0}^{N}\langle f,  (R_{0} \circ \ldots \circ R_{j})^{-1}\phi \rangle_{\Omega},
\end{equation}
where $R_{0} = \mathrm{Id}$. We denote by $\widetilde{\tau}_{-}(x,\theta)$ the (positive) distance to $\partial \widetilde{\Omega}$ in the direction $-\theta$. That is, $\widetilde{\tau}_{-}(x,\theta)$ is the distance from $x \in \widetilde{\Omega}$ to a point on the boundary $\partial \Omega$ along the line $\{ x - t \theta \, | \, t \geq 0 \}$.
\begin{remark}
$\widetilde{\tau}_{-}$ is only necessarily defined for $(x,\theta)$ such that the broken ray through $x$ in the direction $-\theta$ first intersects $E$ at a point $z$ with direction $\eta$ such that $\alpha(z,-\eta) > 0$. Furthermore, if we assume convexity of $\Omega$ so that $\tau_{-}$ is smooth, and if $E$ can be parametrized analytically near $x_{0} \in E$, then $\widetilde{\tau}_{-}$ will be a real analytic function of $(x,\theta)$ in some open set (see Figure \ref{fig:extendedphasedomain}).
\end{remark}
\begin{figure}
\centering
\def \svgwidth{0.8\columnwidth}
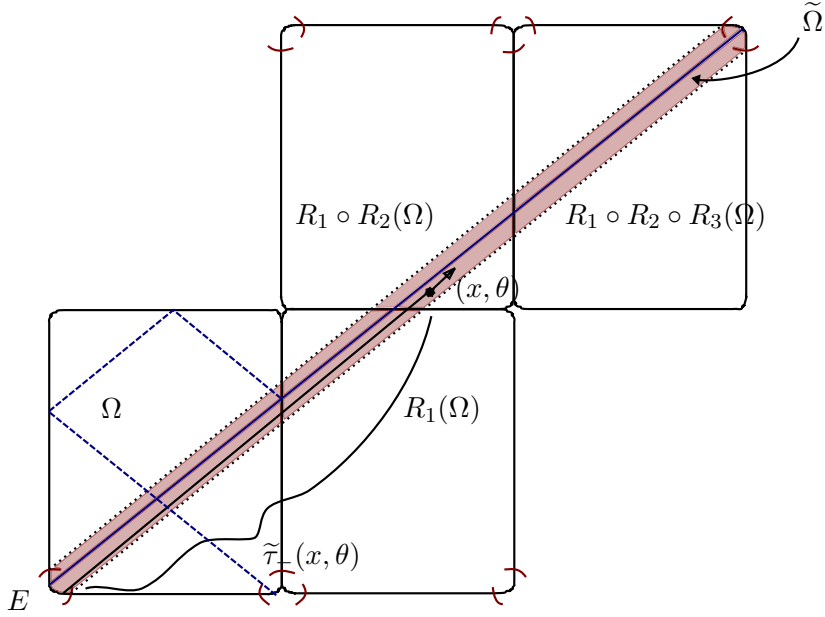
\caption{Illustration of how $\tau_{-}$ is extended to $\widetilde{\tau}_{-}$ on a subset of $\widetilde{\Omega} \times \mathbb{S}^{n-1}$.\label{fig:extendedphasedomain}}
\end{figure}
\begin{figure}
\centering
\def \svgwidth{0.8\columnwidth}
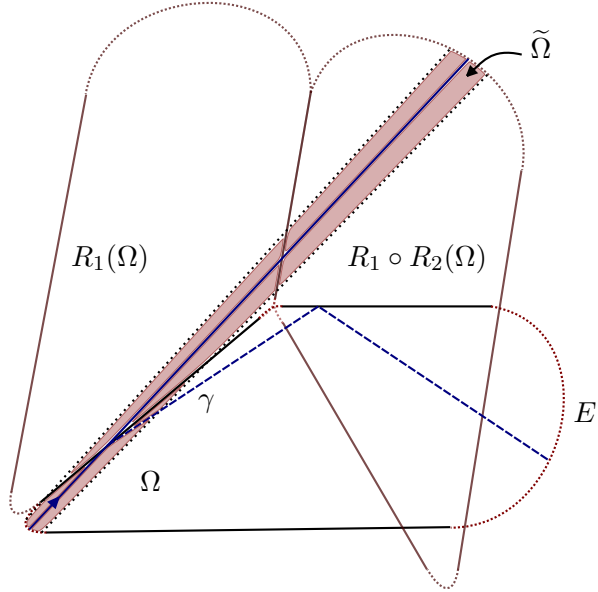
\caption{An example where one cannot reflect the entire domain $\Omega$ in order to construct $\widetilde{\Omega}$; doing so yields overlapping regions. Thus it is important to restrict to a localized beam of broken rays.\label{fig:extendedphasedomain2}}
\end{figure}

Using these constructions, we can write
\begin{equation}
\widetilde{w}(x,\theta) = \exp\left( -\int_{0}^{\widetilde{\tau}_{-}(x,\theta)}\widetilde{\sigma}(x-\tau \theta, \theta)\,d\tau \right), \qquad (x-\widetilde{\tau}_{-}(x,\theta)\theta,\theta) \in \mathrm{supp}(\alpha). 
\end{equation}
\begin{remark}
If $\widetilde{\tau}_{-}$ is analytic and $\sigma$ is constant, then $\widetilde{w}$ is also analytic. Furthermore, if $\sigma \equiv 0$ identically, then $\widetilde{w} \equiv 1$ which is analytic regardless of how $\widetilde{\tau}_{-}$ behaves. This gives some idea about the kinds of domains and attenuation functions which can give rise to an extended weight $\widetilde{w}$ that is analytic. In general, if $\partial \Omega$ is smooth then $\widetilde{w}$ is smooth due to the smoothness of $\tau_{-}$. Note that in \cite{hubenthal2}, for the two dimensional square it was necessary to assume that $\sigma = 0$ near the corners in order to redefine $\tau_{-}$ as the distance to a smooth boundary enclosing the original domain. However, such technicalities are avoided when $\Omega$ has smooth boundary.
\end{remark}

With all of these notations, we can write the broken ray transform for $(x,\theta) \in \mathrm{supp}(\alpha)$ as
\begin{equation}
\alpha(x,\theta) I_{\sigma, E}f(x,\theta) = \alpha(x,\theta) \int_{\mathbb{R}_{+}}\widetilde{w}(x+t\theta, \theta)\widetilde{f}(x + t \theta)\,dt,
\end{equation}
which is identical to a standard x-ray transform of $\widetilde{f}$ on the extended domain $\widetilde{\Omega}$. This will allow us to apply the analytic microlocal techniques of \cite{xraygeneric}. From (\cite{hubenthal2}, Lemma 1), we also have that $\alpha I_{\sigma, E}: L^{2}(\Omega \times \mathbb{S}^{n-1}) \to L^{2}(\Gamma_{-}, d\Sigma)$ is bounded.

For a given $M_{max} \in \mathbb{N}$, we define the \textit{visible set} $\mathcal{M}$ by
\begin{align}
\mathcal{M} & := \{x \in \Omega \, | \, \forall \xi \in \mathbb{S}^{n-1}, (x,\xi) \in N^{*}\gamma \textrm{ for some regular broken ray $\gamma$}\\
& \hspace{7cm} \textrm{ with $N(\gamma) \leq M_{max}$}\}. \notag
\end{align}
We have the following injectivity result for analytic weights $\widetilde{w}$:
\begin{theorem}Let $E \subset \partial \Omega$ be open, $K \Subset \Omega$ a set restricting the support of $f$, $M_{max} \in \mathbb{N}$ be the maximum number of reflections considered, $\mathcal{M}$ be the corresponding visible set with respect to $E$ and $K$, and $\mathcal{M}_{0} \Subset \mathcal{M}$. Also suppose $\sigma$ is such that $\widetilde{w}$ is analytic (e.g. $\sigma \equiv 0$). Then $I_{\sigma, E}$ is injective on $L^{2}(\mathcal{M}_{0})$. \label{thm:injectiveeuc}\end{theorem}

To formulate a stability estimate, we must first parametrize a family of regular broken rays, having an upper bound on the number of reflections, and whose conormal bundles cover $T^{*}\Omega$. Let $K \Subset \Omega$ and let $\alpha$ be a smooth cutoff function on $\Gamma_{-}$ whose support contains only regular broken rays with respect to $E$ and $K$. We define the \textit{microlocally visible set} with respect to $(E, K, \alpha)$ by
\begin{equation}
\mathcal{M}' := \{ (z,\xi) \in T^{*}\Omega \, | \, (z,\xi) \in N^{*}\gamma_{x,\theta} \textrm{ for some }(x,\theta) \textrm{ with }\alpha(x,\theta) > 0\}.
\end{equation}
The set $\mathcal{M}'$ will be useful in recovering some of the singularities of $f$ in cases where injectivity of $I_{\sigma, E}$ does not occur. 

We briefly introduce the normal operator $\mathcal{N}_{\sigma, E, \alpha}$ defined by
\begin{equation}
\mathcal{N}_{\sigma, E, \alpha} := (\alpha I_{\sigma, E})^{*}(\alpha I_{\sigma, E}),
\end{equation}
where $\alpha \in C_{0}^{\infty}(\Gamma_{-})$. Here $(\alpha I_{\sigma, E})^{*}: L^{2}(\Gamma_{-}, d\Sigma) \to L^{2}(\Omega)$ is the adjoint of $\alpha I_{\sigma, E}$ as an operator from $L^{2}(\Omega)$ to $L^{2}(\Gamma_{-}, d\Sigma)$ (restricted to functions with no angular dependence). We will discuss these notions more in \S \ref{sec:stability}, but for now we may state the following stability result analogous to (\cite{xraygeneric}, Theorem 2).
\begin{theorem}
\begin{enumerate}
\item[(a)] Let $K \Subset \Omega$. Fix $\sigma \in C^{2}$ and choose $\alpha \in C^{\infty}(\Gamma_{-})$ to be a smooth cutoff function supported on a collection of regular broken rays with respect to $E$ and $K$ and with at most $M_{max}$ reflections. Fix a set $\mathcal{M}_{0} \Subset \mathcal{M}$ compactly contained in the visible set $\mathcal{M}$ with respect to $(E, K, \alpha)$. If $I_{\sigma, E, \alpha}$ is injective on $L^{2}(\mathcal{M}_{0})$, then
\begin{equation}
\frac{1}{C}\|f\|_{L^{2}(\mathcal{M}_{0})} \leq \| \mathcal{N}_{\sigma, E, \alpha}f \|_{H^{1}(\Omega)} \leq C\|f\|_{L^{2}(\mathcal{M}_{0})}. \label{eq:stabilityestimate}
\end{equation}
\item[(b)] Let $\alpha^{0}$ be as above related to some fixed $\sigma_{0}$. Assume that $I_{\sigma_{0}, E, \alpha^{0}}$ is injective on $L^{2}(\mathcal{M}_{0})$. Then estimate (\ref{eq:stabilityestimate}) remains true for $(\sigma, \alpha)$ in a small $C^{2}$ neighborhood of $(\sigma_{0}, \alpha^{0})$ with a uniform constant $C > 0$.
\end{enumerate}\label{thm:stabilityestimate}
\end{theorem}

\begin{example}As a straightforward example to illustrate Theorem \ref{thm:injectiveeuc}, we consider the case that $\Omega$ is the unit cube $[0,1]^{3}$. Of course, the boundary $\partial \Omega$ is not smooth at its edges and corner points, but similarly as in \cite{hubenthal2} with the square, the special structure of the cube yields a global tiling of $\rn{3}$ via reflection across edges, and the structure of the boundary yields an unambiguous definition of reflections at corner points. Of course, one could modify $E$ to include a neighborhood of all corners/edges and then smooth them to make the example more directly related to the framework of this paper.

Assuming $\sigma \equiv 0$, we have that $\widetilde{w} \equiv 1$ is analytic. If $E$ is an open subset of $\partial \Omega$ which contains $3$ faces meeting at a corner point, and satisfies the condition
\begin{equation}
\forall x \in \Omega, \, \forall \xi \in \mathbb{S}^{n-1}, \, \left( \{z \in \rn{3} \, | \, (z-x) \cdot \xi = 0 \} \cap \partial \Omega \right) \setminus \partial E \neq \emptyset, \label{eq:condition}
\end{equation}
then $\mathcal{M} = \Omega$. In other words, the boundary of $E$ must not lie in any given hyperplane. This is the case if $\partial E$ has suitable curvature (see Figure \ref{fig:cubeexample}). Such a condition ensures that for any unit covector $(x,\xi)$, we can find a normal covector $(x,\theta)$ which intersects $\partial \Omega$ suitably far from $\partial E$, so as to avoid possible singularities of the operator $I_{\sigma,E}$ resulting from jumps in the function $N(z,\eta)$. In Figure \ref{fig:cubeexample} we then have that $I_{0,E}$ is injective for $f \in \mathcal{E}'(\Omega)$. Note that if $E$ does not contain $3$ adjacent faces, then there are simple counterexamples where injectivity does not occur. Specifically, take $E_{0}$ to be an open subset of $\partial \Omega$ containing the $3$ faces adjacent to the lower left corner of the cube (up is the $z$-direction). Then remove a narrow slab to define
\begin{equation*}
E := E_{0} \setminus \{(x_1,x_2,x_3) \in \overline{\Omega} \, | \, x_3 \geq 1-\epsilon \}
\end{equation*}
where $\epsilon > 0$ is arbitrarily small. Then any covectors of the form $(x, (0,0,\lambda))$ with $x_{3} > 1-\epsilon$ are not in $\mathcal{M}'$ since any broken ray whose conormal bundle contains said covector would be a trapped ray that never touches $E$. In this sense, the requirement that $E$ contain $3$ adjacent edges is sharp.
\end{example}
One important point to note is that just as for the usual x-ray transform, in dimensions $3$ or higher it is conceptually much easier to obtain injectivity, since then there are many possible directions normal to a given unit covector. To contrast, in \cite{hubenthal2} one has to be very careful to utilize the geometry of $E$, the square, and the compact support of $f$ to be able to recover all possible wavefront directions at a given point.
\begin{figure}
\centering
\def \svgwidth{0.5\columnwidth}
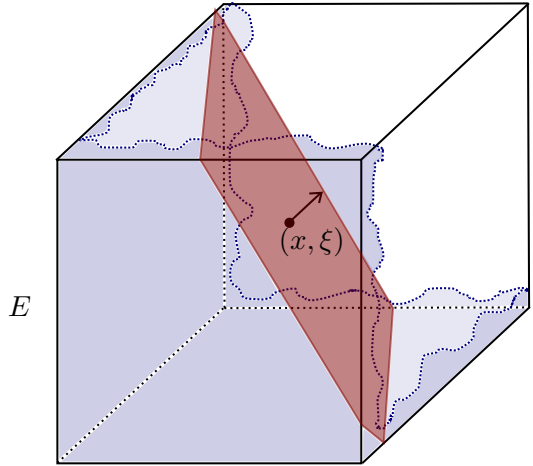
\caption{Illustration of one possible choice of $E$ on the $3$-dimensional cube such that $I_{0,E}$ is injective.\label{fig:cubeexample}}
\end{figure}


\section{Injectivity of $I_{\sigma, E}$ for Analytic Weights $\widetilde{w}$ \label{sec:injective}}
In this section we establish an injectivity result for $I_{\sigma, E}$ using the analytic microlocal approach of \cite{xraygeneric}. As before, let $M_{max} \in \mathbb{N}$ and let $\alpha(x,\theta)$ to be a smooth cutoff function on $\Gamma_{-}$ supported on a set of regular broken rays with at most $M_{max}$ reflections. As can be seen in \cite{hubenthal2} for certain choices of $E$ for the square and depending on the a priori known support of $f$, there exist regular broken rays which reflect at or near boundary points of $E$. However, in dimensions $3$ or higher, as we will see it is generally possible to recover all wavefront covectors of $f$ at a given point $x$ without having to deal with broken rays that reflect at boundary points of $E$.

In future computations, we often assume further that $\alpha = \sum_{k=1}^{m}\alpha_{k}$, and that for all $1 \leq k \leq m$, each regular broken ray in the support of $\alpha_{k}$ reflects against the same sequence of hyperplanes before returning to $E$. This will allow us to get as global an injectivity result as possible for $I_{\sigma, E}$. Moreover, we can then primarily focus on each term separately with respect to $k$, which are easier to understand from a geometric standpoint.

Following the same general approach as in \cite{hubenthal2} for the 2-dimensional square, we consider the case when $\widetilde{w}$ is real-analytic for each cutoff $\alpha_{k}$ (recall that $\widetilde{w}$ depends on the sequence of reflection faces, and hence depends on $\alpha_{k}$). We have the following useful proposition concerning the analytic wavefront set of an arbitrary $f \in \mathcal{E}'(\Omega)$. 
\begin{proposition}
Let $f \in \mathcal{E}'(\Omega)$. Suppose that $I_{\sigma, E}f(x,\theta) = 0$ for all $(x,\theta)$ in a small neighborhood $V$ of $(x_{0},\theta_{0}) \in \Gamma_{-}(E)$, where $\gamma_{x_{0},\theta_{0}}$ is a regular broken ray. Further, suppose that $\widetilde{w}(z,\eta)$ is analytic for all $(z,\eta)$ such that $(z - \widetilde{\tau}_{-}(z,\eta)\eta, \, \eta) \in V$. Then $\mathrm{WF}_{A}(f) \cap N^{*}\gamma_{x_{0},\theta_{0}} = \emptyset$.\label{prop:microlocaleuc}
\end{proposition}
\begin{proof} The proof is very similar to that of (\cite{hubenthal2}, Proposition 1), which deals specifically with the square in $\rn{2}$, but we provide the details here for convenience. We must construct coordinates near the broken ray $\gamma_{0}$. By taking the neighborhood of $(x_{0},\theta_{0})$ to be suitably small, we have that all nearby broken rays $\gamma_{x,\theta}$ hit no corner points and reflect against the same sequence of hyperplanes, which we denote by $\{P_{1},\ldots, P_{N}\}$, where $M = M(x_{0},\theta_{0})$. Let $\{x_{0,1}, \ldots, x_{0,M}\}$ be the sequence of reflection points on $\partial \Omega$ for $\gamma_{x_{0},\theta_{0}}$, so that $x_{0,j} \in P_{j}$. We may write $P_{j} = \{x \in \rn{n}\, | \, (x - x_{0,j}) \cdot \xi_{j} = 0\}$ where $\xi_{j} = \nu(x_{0,j})$ is a unit normal vector to $P_{j}$.

Note that for a general affine hyperplane $P = \{x \in \rn{n} \, | \, (x - a)\cdot \xi = 0\}$, reflection across $P$ is given by
\begin{equation}
R_{P}(x) = R_{\xi, a}(x) =  x + 2\xi (a-x)\cdot \xi.
\end{equation}
As described in \S \ref{sec:mainresults}, we construct a larger $\widetilde{\Omega}$ consisting of $N$ reflected copies of $\Omega$ which are glued together with the original $\Omega$, with the reflections depending on the given sequence of hyperplanes.

Now choose a point $p_{0}$ on the line $x_{0} + t\theta_{0}$ for $t < 0$ so that $p_{0} \notin \widetilde{\Omega}$. Define $x = p_{0} + t\theta$. We use the splitting $y = (y',y^{n})$ for $y \in \rn{n}$ and $(\theta', \theta^{n})$ as the coordinates for $\theta$. Then $(\theta, t)$ are local coordinates near any point of $\widetilde{\Omega} \cap \widetilde{\gamma}_{x_{0},\theta_{0}}$ so long as $|\theta| = 1$ and $|\theta - \theta_{0}| \ll 1$. We can assume without loss of generality that $\theta_{0}' = 0, \, \theta_{0}^{n} = 1$. Write $x = (\theta', t)$. Then $x$ are the coordinates we're looking for defined on
\begin{equation*}
U = \{x = (\theta', t) \, | \, |\theta'| < \epsilon, \, l^{-} < t < l^{+} \} \subset \widetilde{\Omega}.
\end{equation*}
Since $f$ is compactly supported inside of $\Omega$, we may shrink $\epsilon$ as necessary so that we can take $l^{-}, l^{+}$ to be constant. Now let $(z_{0}, \xi_{0} )\in N^{*}\gamma_{0}$ and consider the corresponding reflected covector via path unfolding given by $(\widetilde{z}_{0}, \widetilde{\xi}_{0}) \in N^{*}\widetilde{\gamma}_{0}$. Specifically, let $\{\xi_{1},\xi_{2},\ldots, \xi_{j}\}, \, j \leq M$ be the sequence of normal vectors corresponding to the ordered reflection points of the regular broken ray $\gamma_{x_{0},\theta_{0}}$ up until reaching the point $z_{0}$. Define $\widetilde{z}_{0} = Rz_{0} := R_{1}R_{2} \cdots R_{j}z_{0}$. Also, consider the operators $S_{j}:\rn{n} \to \rn{n}$ defined as reflections across the hyperplanes $z \cdot \xi_{j} = 0$ that pass through the origin. Specifically, $S_{j}v = v - 2(v \cdot \xi_{j})\xi_{j}$. We then define $\widetilde{\xi}_{0} = S\xi_{0} := S_{1}S_{2} \cdots S_{j}\xi_{0}$. The argument of Proposition 1 in \cite{xraygeneric} shows that $(\widetilde{z}_{0}, \widetilde{\xi}_{0}) \notin \mathrm{WF}_{A}(\widetilde{f})$. 

Now we undo the reflection process to conclude that $(z_{0}, \xi_{0}) \notin \mathrm{WF}_{A}(f)$. In particular, we use the respective inverse transformations $S^{-1}$ and $R^{-1}$, generated by $S_{j}^{-1} = S_{j}$ and $R_{j}^{-1} = R_{j}$ to recover $(z_{0},\xi_{0})$ from $(\widetilde{z}_{0}, \widetilde{\xi}_{0})$. Note that the reflections $S_{j}, R_{j}$ preserve the singularities of $f$, and so the analytic wavefront set of $\widetilde{f}$ transforms in the obvious way. Thus $(\widetilde{z}_{0}, \widetilde{\xi}_{0}) = (Rz_{0}, S\xi_{0}) \notin \mathrm{WF}_{A}(\widetilde{f}) \Longrightarrow (z_{0}, \xi_{0}) \notin \mathrm{WF}_{A}(f)$.
\end{proof}

\begin{proof}[Proof of Theorem \ref{thm:injectiveeuc}] Suppose that $I_{\sigma, E}f = 0$. By Proposition \ref{prop:microlocaleuc} we have that $f$ is analytic on $K$. Since $\mathrm{supp}(f) \subset K$, we have that $f$ can be extended to an entire function and hence must be identically zero.\end{proof}

\section{Stability \label{sec:stability}}
We seek to show the stability result of Theorem \ref{thm:stabilityestimate} similarly to how stability was shown in \cite{hubenthal2} for the square. The first step is to analyze the normal operator $\mathcal{N}_{\sigma, E, \alpha} := I_{\sigma, E, \alpha}^{*}I_{\sigma, E, \alpha}$ and deduce that it is in fact a sum of a pseudo differential operator of order $-1$ elliptic on $\mathcal{M}'$ plus an integral operator with $C^{\infty}$ Schwartz kernel. The key idea is to utilize a clever change of variables which depends on each summand in the expansion of $\mathcal{N}_{\sigma, E, \alpha}$. Such a change of variables is based on the idea of unfolding broken rays into straight lines in the corresponding augmented domain $\widetilde{\Omega}$. We will use such simplifications to prove the following proposition:

\begin{proposition}
$\mathcal{N}_{\sigma,E,\alpha} = \mathcal{N}_{\sigma,E,\alpha,ballistic} + \mathcal{N}_{\sigma,E,\alpha,reflect}$ where $\mathcal{N}_{\sigma,E,\alpha,ballistic}$ is a classical pseudo differential operator of order $-1$, elliptic on $\mathcal{M}'$, and $\mathcal{N}_{\sigma, E,\alpha,reflect}$ is an operator with $C_{0}^{\infty}(\Omega \times \Omega)$ Schwartz kernel. Thus there exists a classical pseudo differential operator $Q$ in $\Omega$ of order $1$ such that
\begin{equation}
Q\mathcal{N}_{\sigma, E,\alpha}f = f + Q\mathcal{N}_{\sigma,E,\alpha,reflect}f + \mathcal{S}_{1}f
\end{equation}
for any $f \in L^{2}(\Omega)$, where $\mathcal{S}_{1}$ is microlocally smoothing on $\mathcal{M}'$. \label{prop:normaldecompositioneuc}
\end{proposition}

\subsection{Analyzing the Normal Operator $\mathcal{N}_{\sigma, E, \alpha}$ \label{sec:normalop}}
In order to decompose the normal operator in a useful way, it will be helpful to use the following functions: for $j \in \mathbb{Z}$ we define the variables $z_{j},\theta_{j}$ depending on $x \in \Omega$, $\theta \in \mathbb{S}^{n-1}$ by
\begin{equation*}
(z_{j}, \theta_{j}) := T^{j}(x-\tau_{-}(x,\theta)\theta,\, \theta)
\end{equation*}
Notice that $(z_{0},\theta_{0}) = (x - \tau_{-}(x,\theta)\theta, \theta)$.

Now recall from \cite{hubenthal2} that the adjoint of $\alpha I_{\sigma, E}$ is computed using a special change of variables as well as the invariance of the area form $d\Sigma$ on $\Gamma_{-}$ with respect to the billiard map $T$. For convenience, we repeat some of the details as follows. Observe that for any $g \in L^{2}(\Gamma_{-}, d\Sigma)$
\begin{align}
& \quad \int_{\Gamma_{-}}\alpha(x,\theta) [I_{\sigma,E}f](x,\theta) g(x,\theta)\,d\Sigma \notag\\
& = \int_{\Gamma_{-}}\alpha(x,\theta) g(x,\theta) \sum_{j=0}^{M}\int_{\rn{+}} w_{j}(z_{j}+t\theta_{j},\theta_{j})f(z_{j} + t \theta_{j})|\nu(x) \cdot \theta|\,dt\,dS(x)\,d\theta \label{eq:adjoint0}
\end{align}
We then make the change of variables $(x,\theta,t) \mapsto (y,\eta)$ where $(y,\eta) := (z_{j}(x,\theta) + t \theta_{j}(x,\theta), \theta_{j}(x,\theta))$, depending on each $j$ in the sum, which has the inverse
\begin{align}
(x,\theta) & = (z_{-j}(y,\eta), \theta_{-j}(y,\eta)) \notag \\
t & = \tau_{-}(y,\eta). \label{eq:reflectiondiffeo}
\end{align}
In order to change between $(x,\theta)$ and $(y,\eta)$, we note the identity
\begin{equation*}
(z_{m}(x,\theta) + s\theta_{m}(x,\theta), \theta_{m}(x,\theta)) = (z_{m-j}(y,\eta) + s \theta_{m-j}(y,\eta), \theta_{m-j}(y,\eta)).
\end{equation*}

At this point, we make the observation that given a fixed number of reflections $j \in \mathbb{N}$, each $(y,\eta) \in \Omega \times \mathbb{S}^{n-1}$ corresponds to a unique point $(x,\theta, t) \in \Gamma_{-} \times \mathbb{R}_{+}$ given by $(x,\theta, t) = (T^{-j}(y - \tau_{-}(y,\eta)\eta, \eta), \, \tau_{-}(y,\eta))$. Thus (\ref{eq:reflectiondiffeo}) is a diffeomorphism. We apply the change of variables in two steps. First we change coordinates from $(x,\theta)$ to $(z,\eta) = T^{-j}(x,\theta)$ in (\ref{eq:adjoint0}) and use the fact $T$ preserves the area form $|\nu(x) \cdot \theta| \,dS(x)\,d\theta$ to get
\begin{align}
& \int_{\Gamma_{-}}\alpha(T^{j}(z,\eta)) g(T^{j}(z,\eta)) \nonumber \\
& \cdot \sum_{j=0}^{M}\int_{\rn{+}} w_{j}(z+t\eta,\eta)f(z + t \eta)|\nu(z) \cdot \eta|\,dt\,dS(z)\,d\eta. \label{eq:adjoint1}
\end{align}
Finally we make the change of variables $y = z + t\eta$ to convert the integration against $dt\,dS(z)$ into an integration over $\Omega$, see (\cite{inversesource}, Theorem 1(b)). This yields
\begin{align}
& \sum_{j=0}^{M} \int_{\Omega \times \mathbb{S}^{n-1}}\alpha(T^{j}(y - \tau_{-}(y,\eta)\eta, \, \eta)) g(T^{j}(y - \tau_{-}(y, \,\eta)\eta,\,\eta)) \nonumber \\
& \hspace{2cm} \cdot w_{j}(y,\eta)f(y)dy\,d\eta. \label{eq:adjoint2}
\end{align}
Thus
\begin{equation}
\left( \alpha I_{\sigma, E} \right)^{*}g(x) = \sum_{j=0}^{N} \int_{\mathbb{S}^{n-1}}[\alpha g](z_{j}(x,\theta), \theta_{j}(x,\theta))w_{j}(x,\theta)\,d\theta. \label{eq:adjoint}
\end{equation}

From now on we take $\alpha$ to be a sum of cutoffs $\alpha_{k}$, $k =1, \ldots, m$, such that each $\alpha_{k}$ is supported on a set of regular broken rays which reflect off the same sequence of hyperplanes defining $\partial \Omega$. Moreover, each regular broken ray in the support of $\alpha_{k}$ will have $0 \leq M_{k} \leq M_{max}$ reflections. The normal operator $\mathcal{N}_{\sigma, E, \alpha} = (\alpha I_{\sigma, E})^{*}(\alpha I_{\sigma, E})$ for convex, piecewise smooth Euclidean domains is then
\begin{align*}
& \mathcal{N}_{\sigma,E,\alpha}f(x) \\
& = \sum_{k=1}^{m}\sum_{j_{1}=0}^{M_{k}}\sum_{j_{2}=0}^{M_{k}}\int_{\mathbb{S}^{n-1}}\int_{\mathbb{R}_{+}} |\alpha_{k}(z_{-j_{1}}, \theta_{-j_{1}})|^{2} w_{j_{1}}\left( x,\theta \right) w_{j_{2}}\left(z_{j_{2}-j_{1}}+ t\theta_{j_{2}-j_{1}}, \theta_{j_{2}-j_{1}}\right)\\
& \cdot f\left(z_{j_{2}-j_{1}}+ t\theta_{j_{2}-j_{1}}\right)\, dt\,d\theta.
\end{align*}
We pick off the ballistic terms where $j_{1}=j_{2}$ to decompose $\mathcal{N}_{\sigma,E,\alpha}f$ as a sum
\begin{equation*}
\mathcal{N}_{\sigma,E,\alpha} = \mathcal{N}_{\sigma,E,\alpha, ballistic} + \mathcal{N}_{\sigma,E,\alpha,reflect}
\end{equation*}
with
\begin{align}
& \mathcal{N}_{\sigma,E,\alpha,ballistic}f(x)\notag \\
& = \sum_{k=1}^{m}\sum_{j=0}^{M_{k}}\int_{\mathbb{S}^{n-1}}\int_{\mathbb{R}_{+}} |\alpha_{k}(z_{-j},\, \theta_{-j})|^{2} w_{j}\left( x,\theta \right) w_{j}\left(x + (\tau_{-}(x,\theta)+t)\theta, \theta\right)\notag \\
& \qquad \cdot f\left(x + (\tau_{-}(x,\theta)+t)\theta \right)\, dt\,d\theta \notag \\
& = \sum_{k=1}^{m}\sum_{j=0}^{M_{k}}\int_{\mathbb{S}^{n-1}}\int_{\mathbb{R}} |\alpha_{k}(z_{-j}, \, \theta_{-j})|^{2} w_{j}\left( x,\theta \right) w_{j}\left(x + t\theta, \theta\right)f\left(x + t\theta \right)\, dt\,d\theta \notag \\
& = \sum_{k=1}^{m}\sum_{j=0}^{M_k}\int_{\mathbb{S}^{n-1}}\int_{\rn{+}} \left[ |\alpha_{k}(z_{-j}(x,\cdot), \, \theta_{-j}(x,\cdot))|^{2} w_{j}\left( x,\cdot \right) w_{j}\left(x + t\theta, \cdot \right)\right]_{even}(\theta) \notag \\
& \qquad \cdot f\left(x + t\theta \right)\, dt\,d\theta \label{eq:normalballisticform0}
\end{align}
and
\begin{align}
\mathcal{N}_{\sigma,E,\alpha,reflect}f(x) & = \sum_{k=1}^{m}\sum_{j_{1}=0}^{M_{k}}\sum_{j_{2}=0, j_{2} \neq j_{1}}^{M_{k}}\int_{\mathbb{S}^{n-1}}\int_{\mathbb{R}_{+}} |\alpha_{k}(z_{-j_{1}}, \theta_{-j_{1}})|^{2} \notag\\
& \cdot w_{j_{1}}\left( x,\theta \right) w_{j_{2}}\left(z_{j_{2}-j_{1}}+ t\theta_{j_{2}-j_{1}}, \theta_{j_{2}-j_{1}}\right) \notag\\
& \cdot f\left(z_{j_{2}-j_{1}}+ t\theta_{j_{2}-j_{1}}\right)\, dt\,d\theta.\label{eq:normalreflectform0}
\end{align}
The notation $\left[ g(x,\cdot) \right]_{even}(\theta)$ is the even part of $g$ in the variable $\theta$, defined by $\left[ g(x,\cdot)\right]_{even}(\theta) = \frac{1}{2}\left[ g(x,\theta) + g(x,-\theta) \right]$. 

By Lemma 2 of \cite{xraygeneric} $\mathcal{N}_{\sigma,E,\alpha, ballistic}$ in (\ref{eq:normalballisticform0}) is a classical pseudo differential operator of order $-1$ with principal symbol
\begin{equation}
a_{0}(x,\xi) = 2\pi \sum_{k=0}^{m}\sum_{j=0}^{M_{k}}\int_{\theta \in \mathbb{S}^{n-1}, \, \theta \cdot \xi = 0}|\alpha_{k}(z_{-j}, \theta_{-j})|^{2} \left| w_{j}\left( x,\theta \right)\right|^2 \,d\theta
\end{equation}
The bound on the number of reflections $M_{max}$ ensures that the integral kernel has a positive lower bound on the set of all covectors $(x,\xi)$ such that it is non vanishing at some $\theta \in \mathbb{S}^{n-1}$ normal to $\xi$. Clearly, $\mathcal{N}_{\sigma,E,\alpha,ballistic}$ is also elliptic on the set $\mathcal{M}'$.

\subsection{A Novel Change of Coordinates to Analyze $\mathcal{N}_{\sigma, E, \alpha, reflect}$ \label{sec:changevar}}
In order to have a more complete understanding of $\mathcal{N}_{\sigma,E,\alpha}$, we must understand $\mathcal{N}_{\sigma, E, \alpha, reflect}$, which at first sight seems to be a difficult integral to simplify given the geometry. However, the assumption that all reflections occur on hyperplanes does help immensely. In this section we establish two useful intermediate results that will allow us ultimately to show the smoothing behavior of $\mathcal{N}_{\sigma, E, \alpha}$. Lemma \ref{lemma:reflectideuc} shows that the change of variables $(z, \eta) = T^{-j}(x - \tau_{-}(x,\theta)\theta, \theta)$, $y = z + t\eta$ is equivalent to applying simple changes of variables to $x$ and $\theta$, separately. Lemma \ref{lemma:reflectoutsidebound} then establishes a lower bound on the Jacobian factor introduced by such a change of variables, which ensures that the overall integration is nonsingular.

Recall the operators $R_{\xi,a}:\rn{n} \to \rn{n}$ and $S_{\xi}:\rn{n} \to \rn{n}$, the reflections across the affine plane $(x-a) \cdot \xi = 0$ and linear plane $x \cdot \xi = 0$, respectively. When dealing with an indexed sequence of hyperplanes $\{(\xi_{1}, a_{1}), \ldots, (\xi_{N}, a_{N}) \}$, we also write $R_{j} = R_{\xi_{j},a_{j}}$ and $S_{j} = S_{\xi_{j}}$. We have the following useful result:
\begin{lemma}
Let $x \in \Omega$. Let $V \subset \mathbb{S}^{n-1}$ be an open set such that the sequence of hyperplanes $\{(\xi_{1},a_{1}), \ldots, (\xi_{N},a_{N})\}$ corresponding to $T^{j}(x-\tau_{-}(x,\theta)\theta, \theta)$ are the same for $1 \leq j \leq N$ and for all $\theta \in V$. Then the lines
\begin{equation*}
\{z_{j}+ t\theta_{j} \, | \, t \in \mathbb{R} \} = \{ \pi_{1}\circ T^{j}(x-\tau_{-}(x,\theta)\theta, \theta) + t\pi_{2} \circ T^{j}(x-\tau_{-}(x,\theta)\theta,\theta) \, | \, t \in \mathbb{R} \}
\end{equation*}
and
\begin{equation*}
\{ R_{j} \circ R_{j-1} \circ \cdots \circ R_{1}(x) + t S_{j} \circ S_{j-1}\circ \cdots \circ S_{1}(\theta) \, | \, t \in \mathbb{R} \}
\end{equation*}
coincide.\label{lemma:reflectideuc}
\end{lemma}
\begin{proof}
We proceed inductively. First note that by definition $S_{1}(\theta) = \theta - 2(\theta \cdot \xi_{1})\xi_{1}$ which coincides with $\pi_{2}\circ T(x - \tau_{-}(x,\theta)\theta, \theta)$. It remains to show that $R_{1}(x)+tS_{1}(\theta)$ and $\pi_{1}\circ T(x - \tau_{-}(x,\theta)\theta, \theta) + t\pi_{2}\circ T(x - \tau_{-}(x,\theta)\theta, \theta)$ have a common point. In particular, we will show that they have the same point of intersection with the plane $(x-a_{1})\cdot \xi_{1} = 0$. Note that from the definition of the billiard map, the intersection point of the line $\{ \pi_{1} \circ T^{j}(x- \tau_{-}(x,\theta)\theta, \theta) + t \pi_{2} \circ T^{j}(x-\tau_{-}(x,\theta)\theta,\theta) \, | \, t \geq 0\}$ with the plane $\{x \, | \, (x - a_{j})\cdot \xi_{j} = 0\}$ is the same as that for the line $\{ \pi_{1} \circ T^{j-1}(x-\tau_{-}(x,\theta)\theta,\theta) + t \pi_{2}\circ T^{j-1}(x-\tau_{-}(x,\theta)\theta,\theta) \, |\, t \geq 0\}$.

For the base case, we first note that the intersection point of the line $x + t\theta$ with the plane $(x - a_{1})\cdot \xi_{1} = 0$ is given by
\begin{equation*}
z_{1} = x + \frac{(a_{1} - x)\cdot \xi_{1}}{\theta \cdot \xi_{1}} \theta.
\end{equation*}
Now observe by the definition of $R_{1}$ and $S_{1}$ that
\begin{equation*}
R_{1}(x) + tS_{1}(\theta) = x + 2\left[ (a_{1}-x) \cdot \xi_{1}\right]\xi_{1} + t\left[ \theta - 2(\theta\cdot \xi_{1})\xi_{1}\right].
\end{equation*}
Furthermore, the point of intersection of this line with the plane $(a_{1} - x)\cdot \xi_{1} = 0$ is given by
\begin{align*}
& x + 2\left[ (a_{1}-x)\cdot \xi_{1} \right]\xi_{1} + \frac{(a_{1}-x)\cdot \xi_{1} - 2(a_{1}-x)\cdot \xi_{1}}{\theta \cdot \xi_{1} - 2(\theta \cdot \xi_{1})}\left( \theta - 2(\theta\cdot \xi_{1})\xi_{1}\right)\\
& = x + 2\xi_{1}(a_{1}-x)\cdot \xi_{1} + \frac{(a_{1}-x)\cdot \xi_{1}}{\theta \cdot \xi_{1}}\left( \theta - 2(\theta \cdot \xi_{1})\xi_{1} \right)\\
& = x + \frac{(a_{1}-x) \cdot \xi_{1}}{\theta \cdot \xi_{1}}\theta.
\end{align*}
So the claim holds for $j = 1$.

Now assuming it holds for some $1 \leq k \leq N$, we note that computing the reflection point on the plane $(x-a_{k+1}) \cdot \xi_{k+1} = 0$ of $R_{k+1}\circ \cdots \circ R_{1}(x) + tS_{k+1} \circ \cdots \circ S_{1}(\theta)$ is equivalent to computing that of $R_{k+1}(z_{k}) + tS_{k+1}(\theta_{k})$. By induction, such a reflection point coincides with that of 
\begin{align*}
& \quad \pi_{1}\circ T(z_{k}, \theta_{k}) + t\pi_{2} \circ T(z_{k}, \theta_{k}) \\
= & \quad \pi_{1} \circ T^{k+1}(x -\tau_{-}(x,\theta)\theta,\theta) + t \pi_{2}\circ T^{k+1}(x-\tau_{-}(x,\theta)\theta,\theta).
\end{align*}
This completes the proof.\end{proof}

\begin{lemma}
Let $K \Subset \Omega$ with $\mathrm{dist}(K, \partial \Omega) \geq \epsilon > 0$. Let $\gamma = \bigcup_{j=0}^{N}\gamma_{j}$ be a regular broken ray consisting of line segments $\gamma_{j}$ parametrized such that
\begin{equation*}
(\gamma_{j}(0), \dot{\gamma}_{j}(0)) = T^{j}(x - \tau_{-}(x,\theta)\theta, \theta).
\end{equation*}
Suppose $(x,\theta) = (\gamma_{k}(t_{0}), \gamma_{k}'(t_{0}))$ for some $0 \leq k < N$ and $0 \leq t \leq L(\gamma_{k})$.

Then for any $k < l \leq N$, we have
\begin{equation}
|R_{l} \circ R_{l-1} \circ \cdots \circ R_{k+1}(x) - y | \geq \epsilon
\end{equation}
for all $y \in K$.
\label{lemma:reflectoutsidebound}
\end{lemma}
\begin{proof}In the case of a single reflection (i.e. $l-k= 1$), it is easy to see by the way $\gamma_{j}$ are parametrized that $R_{l}(x) = R_{k+1}(x) = \gamma_{k+1}(-\tau_{+}(x,\theta))$. Note that we have extended each $\gamma_{j}$ to a full line without shifting the parametrization. Let $\{L_{0}, L_{1}, L_{2}, \ldots, L_{N}\}$ be the lengths of $\{\gamma_{0} \cap \Omega, \ldots, \gamma_{N} \cap \Omega\}$. For any $k < l \leq N$, we then have by the parametrizations of $\gamma_{l}$ that
\begin{equation}
R_{l} \circ \cdots \circ R_{k+1}(x) = \gamma_{l}\left( -\tau_{+}(x,\theta) - \sum_{j=k+1}^{l-1}L_{j} \right). \label{eq:reflectedxoutside}
\end{equation}

Since each $\gamma_{j}$ is pointing inward on the boundary $\partial \Omega$ at $t=0$ and $\Omega$ is convex, we have by (\ref{eq:reflectedxoutside}) that $R_{l} \circ \cdots \circ R_{k+1}(x)$ lies outside of $\Omega$. Therefore
\begin{equation*}
|R_{l} \circ \cdots \circ R_{k+1}(x) - y |  \geq d(y, \partial \Omega) \geq \epsilon.
\end{equation*}
\end{proof}

\begin{figure}
\centering
\def \svgwidth{0.8\columnwidth}
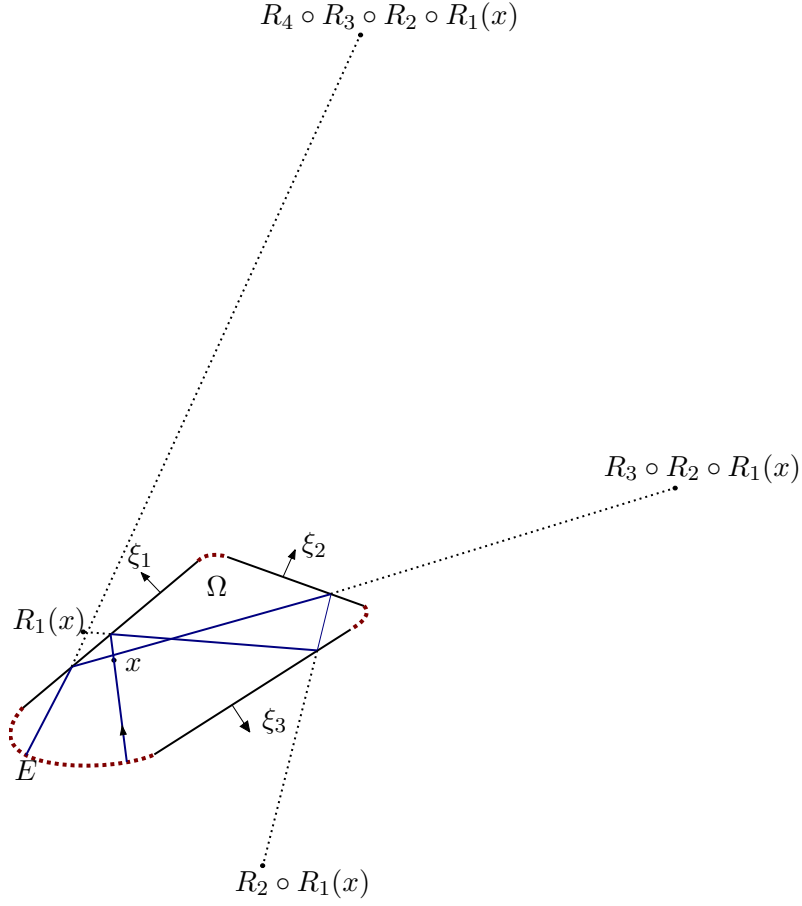
\caption{Illustration of change of variables used to simplify $\mathcal{N}_{\sigma, E, \alpha, reflect}$ in the proof of Proposition \ref{prop:normaldecompositioneuc}. \label{fig:reflectionscheme}}
\end{figure}

We are now ready to prove the main result of this section.
\begin{proof}[Proof of Proposition \ref{prop:normaldecompositioneuc}]
It remains to analyze the part $\mathcal{N}_{\sigma, E, \alpha, reflect}$ of the normal operator. Choose a single term in the sum (\ref{eq:normalreflectform0}) that defines\\ $\mathcal{N}_{\sigma, E,\alpha,reflect}$. That is, fix $k$ and fix $j_{1} \neq j_{2}$. For now we will assume that $j_{2} > j_{1}$, although the other case is virtually the same due to the reversability of the billiard map. We have an integral of the form
\begin{align*}
I_{k,j_{1},j_{2}} & = \int_{\mathbb{S}^{n-1}}\int_{\mathbb{R}_{+}}|\alpha_{k}(z_{-j_{1}}, \theta_{-j_{1}})|^{2}\\
& \cdot w_{j_{1}}\left(x,\theta\right)w_{j_{2}}(z_{j_{2}-j_{1}} + t\theta_{j_{2}-j_{1}}, \theta_{j_{2}-j_{1}})f(z_{j_{2}-j_{1}} + t\theta_{j_{2}-j_{1}})\,dt \, d\theta.
\end{align*}
Using Lemma \ref{lemma:reflectideuc} we have that integration of $f$ over the line $z_{j_{2}-j_{1}} + t\theta_{j_{2}-j_{1}}$ is the same as integration of $f$ over the line
\begin{equation*}
R(x) + tS(\theta) := R_{j_{2}-j_{1}}\circ \cdots \circ R_{1}(x) + t S_{j_{2}-j_{1}}\circ \cdots \circ S_{1}(\theta).
\end{equation*}
We then make the change of variables $\eta = S(\theta)$ to get
\begin{align*}
I_{k,j_{1},j_{2}} &= \int_{\eta \in \mathbb{S}^{n-1}}\int_{t \in \mathbb{R}_{+}}|\alpha_{k}\left(z_{-j_{1}}(x, S^{-1}(\eta)), \theta_{-j_{1}}(x,S^{-1}(\eta)\right)|^{2}\\
& \cdot w_{j_{1}}\left(x,S^{-1}(\eta)\right)w_{j_{2}}(R(x) + t\eta, \eta)f(R(x) + t\eta)\,dt\,d\eta.
\end{align*}
Finally, we make the change to cartesian coordinates by letting $y = R(x) + t\eta$, so that $\eta = \widehat{y - R(x)}$ and $|R(x)-y|^{-n+1}dy = dt\,d\eta$. We obtain
\begin{align*}
I_{k,j_{1},j_{2}} & = \int_{\rn{n}}|\alpha_{k}\left(z_{-j_{1}}\left(x, \, S^{-1}\left(\widehat{y-R(x)}\right)\right), \theta_{-j_{1}}\left(x, \, S^{-1}\left(\widehat{y-R(x)}\right)\right)\right)|^{2}\\
& \cdot w_{j_{1}}\left(x,S^{-1}\left(\widehat{y-R(x)}\right)\right)w_{j_{2}}\left(y, \widehat{y-R(x)}\right)f(y)|y-R(x)|^{-n+1}\,dy.
\end{align*}
By Lemma \ref{lemma:reflectoutsidebound}, the denominator $|y-R(x)|^{n-1}$ is strictly bounded away from $0$, and hence $I_{k,j_{1},j_{2}}$ is a smooth function of $x$.

Finally, let $Q$ be a microlocal parametrix for $\mathcal{N}_{\sigma, E, \alpha, ballistic}$, which is a pseudo differential operator of order $1$ that is elliptic on $\mathcal{M}'$. Then by definition
\begin{equation*}
Q\mathcal{N}_{\sigma, E,\alpha}f = Q\mathcal{N}_{\sigma, E, \alpha, ballistic}f + Q\mathcal{N}_{\sigma, E, \alpha, reflect}f = f + \mathcal{S}_{1}f + Q\mathcal{N}_{\sigma, E, \alpha,reflect}f,
\end{equation*}
where $\mathcal{S}_{1}$ is microlocally smoothing on $\mathcal{M}'$.
\end{proof}

\begin{remark}If we restrict to $f \in L^{2}(K)$ where $K \Subset \mathcal{M}$, $\mathcal{M}$ being the visible set, then $\mathcal{S}_{1}$ will be a smoothing operator on $K$. Furthermore, we can use the stability estimate of \S \ref{sec:stability} to establish injectivity of $I_{\sigma, E,\alpha}$ for $C^{2}$ perturbations of $\sigma$ from constant.
\end{remark}

\begin{proposition}Under the conditions of Theorem \ref{thm:stabilityestimate}, without assuming that $I_{\sigma, E, \alpha}$ is injective,
\begin{enumerate}
\item[(a)] one has the a priori estimate
\begin{equation*}
\|f\|_{L^{2}(K)} \leq C\|\mathcal{N}_{\sigma, E, \alpha}f\|_{H^{1}(\Omega)} + C_{s}\|f\|_{H^{-s}(\Omega)}, \quad \forall s;
\end{equation*}
\item[(b)]  $\mathrm{Ker}{\, I_{\sigma, E, \alpha}}$ is finite dimensional and included in $C^{\infty}(K)$.
\end{enumerate} \label{prop:aprioristability}
\end{proposition}
\begin{proof}The proof exactly the same as that for Proposition 3 of \cite{hubenthal2}, except now we apply the new structure result for the normal operator that is given by Proposition \ref{prop:normaldecompositioneuc}.
\end{proof}
We also remark that smoothness of the kernel of the geodesic X-ray transform of 2-tensors on compact simple Riemannian manifolds with boundary is considered closely in \cite{chappa}.

\subsection{Reducing the Smoothness Condition on $\sigma$ \label{sec:reducingsmoothness}}
Given a choice of smooth cutoff $\alpha$ and a smooth $\sigma$ such that $I_{\sigma, E, \alpha}$ is injective, we would like to be able to perturb $\alpha$ and $\sigma$ in $C^{2}$ and still have $\mathcal{N}_{\sigma, E,\alpha}$ be injective on $L^{2}(K)$ for some $K \Subset \Omega$. We do this according to the following modified version of (\cite{xraygeneric}, Proposition 4). The proof is the same as that given for Proposition 4 of \cite{hubenthal2}, so we omit it.
\begin{proposition}Assume that $\sigma, \alpha$ are fixed and belong to $C^{2}$. Let $(\sigma', \alpha')$ be $O(\delta)$ close to $(\sigma, \alpha)$ in $C^{2}$. Then there exists a constant $C > 0$ that depends on an a priori bound on the $C^{2}$ norm of $(\sigma, \alpha)$ such that
\begin{equation}
\left\| (\mathcal{N}_{\sigma', E, \alpha'} - \mathcal{N}_{\sigma, E, \alpha})f \right\|_{H^{1}(\Omega)} \leq C \delta \|f\|_{L^{2}(K)}. 
\end{equation}\label{prop:normalperturb}
\end{proposition}

Using Proposition \ref{prop:normalperturb} we now have all the pieces required to prove the stability estimate of Theorem \ref{thm:stabilityestimate} stated in \S \ref{sec:mainresults}. 
\begin{proof}[Proof of Theorem \ref{thm:stabilityestimate}] The proof is essentially the same as that for Theorem 2 of \cite{hubenthal2} combined with our more general structure result for $\mathcal{N}_{\sigma, E, \alpha}$ in $n$-dimensional Euclidean domains given by Proposition \ref{prop:normaldecompositioneuc}.
\end{proof}

\section{Conclusion\label{sec:conclusion}}
Altogether, this work provides a generalization to higher dimensions of the results in \cite{hubenthal2}. The essential ingredient was a change of variables in order to simplify $\mathcal{N}_{\sigma, E, \alpha, reflect}$, which is the more nontrivial part of the normal operator. The flatness condition on the reflecting parts of the boundary ensures that there are a countable number of unique sequences of reflecting faces for a given broken ray. If we further impose a limit on the total number of reflections, then there are only finitely many ways in which a broken ray can reflect. The main advantage one has in higher dimensions is that there are far more possible broken rays which can detect a given microlocal singularity. As such, injectivity is easier to demonstrate for $n \geq 3$ for more general choices of $E$. Finally, it is the author's opinion that the technique used in \S \ref{sec:changevar} can be generalized to the case of Riemannian manifolds where the reflecting part of the boundary is flat. This would be an interesting direction for future work.

\section*{Acknowledgments} Support by the Institut Mittag-Leffler (Djursholm, Sweden) is gratefully acknowledged, as many of the ideas in the work were conceived there. This research was also conducted with partial support from the Academy of Finland while at the University of Jyv\"askyl\"a. The author would like to thank Jan Boman for helpful advice in relation to this problem as well as Mikko Salo for providing helpful discussion and feedback. Finally, much thanks is due to the referees for providing helpful comments and suggestions.

\nocite{*}
\bibliographystyle{plain}
\bibliography{references}

\end{document}